\documentclass[12pt,oneside]{amsart}
 
\usepackage{geometry} 
\usepackage{amsmath, amsthm, amssymb} 
\usepackage{graphicx} 
\usepackage{mathrsfs} 
\usepackage{color}
\usepackage{verbatim}
\usepackage[bbgreekl]{mathbbol}
\usepackage[dvipsnames]{xcolor}
\usepackage{enumitem}
\usepackage{graphicx}

\usepackage{geometry}
\makeatletter
\@addtoreset{equation}{section}
\makeatother

\usepackage[all]{xy}

\title{On non-rational fibers of del Pezzo fibrations over curves}
\author{Konstantin Loginov}
\thanks{Partially supported  by the Russian Academic Excellence Project '5-100', Foundation for the Advancement of Theoretical Physics and Mathematics ``BASIS'', and the Simons Foundation.}
\date{} 

\newcounter{cthm}

\newtheorem{proposition}[equation]{Proposition}
\newtheorem{thm}[equation]{Theorem}

\newtheorem{lem}[equation]{Lemma}
\theoremstyle{definition}
\newtheorem{defin}[equation]{Definition}
\newtheorem{remark}[equation]{Remark}
\newtheorem{question}[equation]{Question}
\newtheorem{construction}[equation]{Construction}
\newtheorem{computation}[equation]{Computation}
\newtheorem{exam}[equation]{Example}
\theoremstyle{exam}

\newcommand{\OOO}{\mathscr{O}}

\newcommand{\Addresses}{{
  \bigskip
  \footnotesize
  \textsc{Laboratory of Algebraic Geometry, Faculty of Mathematics \\ National Research University Higher School of Economics and \\ Independent University of Moscow }\\
  \textit{E-mail:} \texttt{kostyaloginov@gmail.com}
}}
\begin{document}
\maketitle
\begin{abstract}
We consider threefold  del Pezzo fibrations over a curve germ whose central fiber is non-rational. Under the additional assumption that the singularities of the total space are at worst ordinary double points, we apply a suitable base change and show that  there is a 1-to-1 correpspondence between such fibrations  and certain non-singular del Pezzo fibrations equipped with a cyclic group action.
\end{abstract}

\section*{Introduction}

It is classically known that a cubic del Pezzo surface can degenerate into a cone over an elliptic curve in a non-singular family. We investigate when a del Pezzo surface can degenerate into a non-rational surface in a ``reasonably good" family. By such family we mean a del Pezzo fibration in the sense of the Minimal Model Program (the MMP for short), see Definition \ref{dPfibration}. In particular, the total space of the fibration should have at worst terminal singularities. The main invariant of such fibrations is the degree $K_{X_\eta}^2$ of its general fiber. Since the general fiber is non-singular, $1\leq K_{X_\eta}^2\leq 9$. Our question is local, so we consider fibrations over curve germs. 

The motivation for the problem comes from the three-dimensional MMP. If we apply the MMP to a (non-singular) rationally connected threefold $U$ over the field of complex numbers, we obtain a variety $X$ birational to $U$ such that it admits a Mori fiber space structure. That is, there is a morphism $\pi: X \longrightarrow B$ with connected fibers, $\pi$-ample anti-canonical class $-K_X$ and $\mathrm{dim}\ B<\mathrm{dim}\ X$. If $\mathrm{dim}\ B=0$ then $X$ is a Fano variety. The rationality problem for (singular) Fano threefolds is far from complete solution, although much is known in the non-singular case, see \cite[Chapter 12]{IP-1999}. If $\mathrm{dim}\ B=2$ then $\pi$ is called a $\mathbb{Q}$-conic bundle. Its fibers are trees of rational curves. In this case the rationality problem for the fibers of $\pi$ is trivial. We work with the case $\mathrm{dim}\ B=1$ which is called a del Pezzo fibration. Its general fiber is rational. But a special fiber can be non-rational. It is easy to show that such fiber is a surface which is birationally ruled over a curve $C$ of genus $g(C)>0$.

In this paper we show that the properties of such del Pezzo fibrations that contain a non-rational fiber, for example the value of $g(C)$, depend on $K_{X_\eta}^2$ and on singularities of $X$. In Proposition~\ref{DegreeBound} we prove that if $X$ is non-singular (respectively, terminal Gorenstein) then $K_{X_\eta}^2 \leq 3$ (resp., $\leq 4$) and the non-rational fiber is a cone over an elliptic curve. This fact is rather elementary and follows from the classification of Gorenstein del Pezzo surfaces \cite{HW-1981}. As mentioned in Remark \ref{reduced}, in the terminal Gorenstein case any fiber is reduced and irreducible, and moreover, a non-rational fiber is necessarily normal. On the other hand, in the non-Gorenstein terminal case, multiple fibers are possible. However, their multiplicity is bounded by $6$ as shown in \cite{MP-2009}.  

In Theorem~\ref{smoothclassification} we use the base change construction to show that in the non-singular case such del Pezzo fibrations with a non-rational fiber are in $1$-to-$1$ correspondence with non-singular $\bbmu_n$-del Pezzo fibrations with certain properties. 

This shows that the non-rational fibers of terminal Gorenstein del Pezzo fibrations form a very restricted class. On the other hand, if we allow $X$ to have worse than terminal singularities then the non-rational fibers are not bounded, see Example \ref{degeneration}. We also give examples of terminal fibrations whose special fiber is birationally ruled over a curve $C$ of genus $g(C)=2,3,4$. It is not known whether one can achieve $g(C)>4$ in this setting, see Question \ref{Bquestion}.

Then we consider the fibrations with very mild singularities, the ordinary double points. Using the base change construction, we classify such fibrations with non-rational central fiber in terms of certain $\bbmu_n$-del Pezzo fibrations, see Theorem \ref{sing_classification}. It appears that in this case $K_{X_\eta}^2=1$ or $4$.

For other results on rationality in families see \cite{KT-2017}, \cite{T-2016}, \cite{P-2017} and references therein. For the classification of non-rational del Pezzo surfaces see \cite{HW-1981}, \cite{Fuji-1995}.

The author is grateful to Yuri Prokhorov for numerous useful discussions, to Alexander Kuznetsov, Dmitry Mineev and Constantin Shramov for their valuable suggestions, and to J\'er\'emy Blanc for posing Question~\ref{Bquestion}.

\section{Preliminaries}

We work over the field of complex numbers. We use terminology and notation of the Minimal Model Program (e.g., \cite{Matsuki2002}, \cite{KMM-1987}).

\begin{defin}
\label{dPfibration}
Let $X$ be a three-dimensional normal projective variety with at worst terminal $\mathbb{Q}$-factorial singularities and let $B$ be a non-singular curve. Then $\pi \colon X \longrightarrow B$ is called a \emph{del Pezzo fibration} (resp., a \emph{weak del Pezzo fibration}) if the following conditions hold:
\begin{enumerate} 
\item
$\pi$ is projective and has connected fibers;
\item
$-K_X$ is $\pi$-ample (resp., $\pi$-nef and $\pi$-big);
\item
$\pi$ is an extremal contraction, that is $\rho (X / B) = 1$.
\end{enumerate}

\emph{The degree of a (weak) del Pezzo fibration} is the degree of its general fiber $X_\eta$. Since $X$ is terminal $X_\eta$ is a non-singular del Pezzo surface. 
\end{defin}

We say that a del Pezzo fibration $\pi: X \longrightarrow B$ is \emph{non-singular} (resp., \emph{Gorenstein}) if so it its total space $X$. If in the above definition $X$ is an complex analytic space and $\pi$ is a proper map, we call $\pi \colon X \longrightarrow B$ an \emph{analyitc del Pezzo fibration}. When we consider $X$ as a germ over $o \in B$ we use the notation $$\pi: X\longrightarrow B \ni o.$$

Let $G$ be a group acting on the fibration $\pi$. Then one can define a \emph{$G$-del Pezzo fibration} as in Definition \ref{dPfibration} with the following modifications: we require $X$ to be $G\mathbb{Q}$-factorial (that is, every $G$-invariant Weil divisor is $\mathbb{Q}$-Cartier) and have $\rho^G (X / B) = 1$. In this paper we will work with $\bbmu_n$-del Pezzo fibrations where $\bbmu_n$ is the cyclic group of order $n$. We fix a primitive root of unity of degree $n$ and denote it by $\zeta_n$. 


\begin{remark}
\label{reduced}
Let $\pi: X \longrightarrow B\ni o$ be a Gorenstein del Pezzo fibration. Consider the fiber $F=\pi^{-1}(o)$. Since $\rho(X/B)=1$ the fiber $F$ is irreducible. Since $X$ is Gorenstein $F$ is reduced \cite[5.1]{Kaw-1988}. Assume that $F$ is non-rational. Then $F$ is normal \cite{Re-1994, AF-2003}. 
\end{remark}

\begin{proposition}
\label{DegreeBound}
Let $\pi:~X \longrightarrow~B\ni o$ be a Gorenstein del Pezzo fibration such that the fiber $F=\pi^{-1}(o)$ is non-rational. Then $F$ is a generalised cone over an elliptic curve and $K_F^2\leq 4$. Moreover, if $X$ is non-singular then $K_F^2\leq 3$.
\end{proposition}
\begin{proof}
The first claim follows from the classification of Gorenstein del Pezzo surfaces, see for example \cite{HW-1981}. Notice that $F$ has only one simple elliptic singularity $x_0$. Let $\phi: T \longrightarrow F$ be the minimal resolution. We have $K_T = \phi^* K_F - E_0$ where $E_0$ is an elliptic curve. Thus, $K_T^2 = K_F^2 + E_0^2 = d + E_0^2$. By the Noether formula $K_T^2 + \chi_{\mathrm{top}} (T) = 12 \chi (\OOO_T) = 0$, and $\chi_{\mathrm{top}} (T) = 0$ since $T$ is a ruled surface over an elliptic curve, so $d = -E_0^2$. On the other hand, by \cite[4.57]{KM-1998} the dimension of the tangent space at $x_0$ to $F$ is equal to $\mathrm{max} (3, -E_0^2)$. If $X$ is Gorenstein it has hypersurface singularities, hence $-E_0^2=d\leq 4$. If $X$ is non-singular, $-E_0^2=d\leq 3$, so we are done.
\end{proof}

The next example shows that the case $d=4$ occurs.

\begin{exam}
Let $X$ be given by the equations 
\begin{align*}
x_1^2 + x_2^2 + x_3^2 + x_4^2 + tx_5^2 &= 0, \\
a_1 x_1^2 + a_2 x_2^2 + a_3 x_3^2 + a_4 x_4^2 + t x_5^2 &= 0 
\end{align*}
in $\mathbb{P}^4\times \mathbb{A}^1_t$ where $a_i \in \mathbb{C}$. One checks that for a general choice of $a_i$ the threefold $X$ has one $cA_1$ singularity, and the fiber $F$ over $0\in \mathbb{A}^1_t$ is a cone over an elliptic curve.
\end{exam}

There are examples of non-Gorenstein fibrations with a non-rational fiber that is birationally ruled over the curve $C$ with $g(C)>1$. 

\begin{exam}
\begin{enumerate}
\item
$ X = ( f_6(x, y, w) + t z^3 = 0 ) \subset \mathbb{P}(1,1,2,3)\times \mathbb{A}^1_t$
where $(x,y,z,w)$ have the weights $(1,1,2,3)$, the polynomial $f_6$ has degree $6$ and is general. The morphism $\pi: X \longrightarrow B = \mathbb{A}^1_t$ is induced by the projection to the second factor. Notice that $X$ has one terminal singularity of type $\frac{1}{2}(1,1,1)$. A general fiber is a degree $1$ del Pezzo surface. The central fiber $F$ is a cone over a hyperelliptic curve $C$ of genus $2$. 
\
\item
$ X = ( f_4(x, y, z) + t w^2 = 0 ) \subset \mathbb{P}(1,1,1,2)\times \mathbb{A}^1_t$
where $(x,y,z,w)$ have the weights $(1,1,1,2)$. Notice that $X$ has one terminal singularity of type $\frac{1}{2}(1,1,1)$. A general fiber is a degree $2$ del Pezzo surface. The central fiber $F=\pi^{-1}(0)$ is a cone over a plane quartic curve $C$, so $g(C)=3$.
\
\item
$ X = ( f_6(x, y, z) + t w^2 = 0 ) \subset \mathbb{P}(1,1,2,3)\times \mathbb{A}^1_t$
where $(x,y,z,w)$ have the weights $(1,1,2,3)$. Notice that $X$ has one terminal singularity of type $\frac{1}{3}(1,1,2)$. A general fiber is a degree $1$ del Pezzo surface. The central fiber $F$ is a cone over a trigonal curve $C$ of genus $4$.
\end{enumerate}
\end{exam}

In the above examples $F$ is normal. However, if we take a special polynomial $f_i$ we can get a non-normal and non-rational fiber. This contrasts with the Gorenstein case. The following natural question was posed by J. Blanc: 

\begin{question}
\label{Bquestion}
Is there a del Pezzo fibration $\pi: X \longrightarrow B$ such that its fiber is birationally ruled over a curve $C$ with $g(C) > 4$?
\end{question}

At the moment, the answer to this question is not known. Terminal singularities is an important restriction as the following example shows.

\begin{exam}
\label{degeneration}
For a moment we consider a fibration that has worse than terminal singularities. Define $\pi: X \longrightarrow B$ as follows: $$ X = ( f_n(x, y, z) + t w = 0 ) \subset \mathbb{P}(1,1,1,n)\times \mathbb{A}^1_t$$ where the coordinates $x, y, z, w$ have the weights $(1,1,1,n)$ and the polynomial $f_n$ is general and has degree $n$. Clearly, $X$ has one singular point of type $\frac{1}{n}(1,1,1)$. In particular, $X$ is log terminal. A general fiber is isomorphic to $\mathbb{P}^2$. The fiber over $t=0$ is a cone over a plane curve of degree~$n$. One can construct similar (log terminal) degenerations to a cone over a curve of arbitrarily large genus in del Pezzo fibrations of any degree $1 \leq d \leq 9$, see \cite[3.9]{K-2013}.
\end{exam}

\section{Non-singular fibrations}
\label{constr}
Let $\pi: X \longrightarrow B\ni o$ be a non-singular del Pezzo fibration such that the fiber  $F=\pi^{-1}(o)$ is non-rational. Then $K_F^2\leq 3$ by Proposition \ref{DegreeBound}. We start with the description of the base change construction.

\

\begin{construction} 
\label{ns_construction}
Let $x_0$ be the (simple elliptic) singularity of $F$. By  \cite[4.57]{KM-1998} there exists a weighted blow-up $\psi: Z \longrightarrow X$ of $x_0 \in X$ with the weights $(c_1, c_2, c_3)$ for some $c_i$ such that $F_Z=\psi^{-1}_*F$ is the minimal resolution of $F$. We have $K_{F_Z}=\psi|_{F_Z}^* F - E|_{F_Z}$. In this case $E|_{F_Z}$ is reduced irreducible non-singular elliptic curve, call it $C$. Notice that $F_Z = \psi^* F - nE$ for  $n\geq 2$, and $E\simeq \mathbb{P}(c_1, c_2, c_3)$. Then $$K_Z = \psi^* K_X + (n-1)E, \ n=c_1+c_2+c_3.$$ After the blow-up $\psi$ the threefold $Z$ may obtain some number of cyclic quotient singularities. However, $F_Z$ does not pass through them. Indeed, let $z_0$ be a singular point on $Z$ and suppose that $z_0\in F_Z$. Since $z_0$ is a cyclic quotient singularity, $\mathbb{C}^3$ covers an analytic neighbourhood $U$ of $z_0$. This covering induces an unramified covering of $F_Z \cap U - \{ z_0 \}$. But $F_Z$ is non-singular, hence $\pi_1(F_Z \cap U - \{ z_0\}) = 0$. This is a contradiction.
 
Now we make a base change. Pick a local coordinate $t$ at the point $o \in B$ and consider the following commutative diagram: 
\[
\xymatrix{
W \ar@{->}[r]^{h} \ar[d]^{\pi_W} & Z \ar[d]_{\pi} 
\\
B'  \ar[r]^{\alpha} & B
}
\]
where $B'\simeq B$, $\alpha: t \mapsto t^n$ and $W$ is the normalization of $Z\times_B B'$. At a general point of $E$ the threefold $Z$ is isomorphic to $$\mathrm{Spec}\ \mathbb{C}[x, y, z, t]/(t-z^n),$$ and the fiber $\pi_Z^{-1}(o)$ is given by $(t=0)$. After the base change we have $$\mathrm{Spec}\ \mathbb{C}[x, y, z, t]/(t^n-z^n)$$ which is singular in codimension $1$. After the normalization we see that $h$ is \'etale in the neighbourhood of a general point of $E_W:=h^{-1}(E)$. Similarly, one can check that the morphism $h$ is ramified along $F_W:=h^{-1}(F_Z)$ and at all the singular points of $Z$.


The fiber $\pi^{-1}_W(o)$ is reduced. However, it is reducible: $\pi^{-1}_W(o) = F_W + E_W$, where $E_W$ covers $E$, and ~$F_W$ is isomorphic to $F_Z$ via $h$. More precisely, $h|_{E_W}$ is totally ramified at $E_W \cap~F_W =: C_W$. It follows that $F_W$ is non-singular. Moreover, $F_W$ and $E_W$ intersect transversally. The Galois group $\bbmu_n$ of $h$ acts on $W$ preserving the central fiber. 
We make a $\bbmu_n$-equivariant contraction of $F_W$ (see computation below) and get a $\bbmu_n$-del Pezzo fibration $\pi_V: V \longrightarrow B \ni o$ with a rational central fiber. All these maps are shown in the following diagram:
\begin{equation}
\vcenter{
\label{basechange}
\xymatrix{
F_W + E_W \subset W \ar[r]^{h} \ar[d]^{\tau} & \ar[d]_{\psi} Z \supset F_Z + n E  
\\
E_V \subset V \ar[d]^{\pi_V} & X \supset F \ar[d]_{\pi}
\\
B'  \ar[r]^{\alpha} & B 
}}
\end{equation}
\end{construction}
\

\begin{computation}  
As before, $\phi: T \longrightarrow F$ is the minimal resolution. Denote by $f_T$ a ruling of $T$, and by $f_Z$ a ruling of $F_Z$, put $f:=\psi(f_Z)$. We need the following formulas.
\begin{align*}
K_F \cdot f &= \phi^*K_F \cdot \phi^* f 
=\phi^* K_F \cdot f_T \\ 
&= ( K_{T} + E_0 ) \cdot f_T 
=-2 + 1 = -1,
\end{align*}
\begin{align*}
K_Z \cdot f_Z &= ( \psi^* K_X + (n-1) E ) \cdot f_Z \\
&=K_X \cdot f + n-1 \\
&= K_F \cdot f + n-1 = n-2. 
\end{align*} 

We want to contract $F_W$. We calculate $K_W\cdot f_W $ where $f_W$ is a ruling of $F_W\simeq F_Z$. Since $h$ is totally ramified along  $F_W$ by the Hurwitz formula we have  
$$K_W = h^* K_Z + (n-1) F_W.$$ 

Since $( F_W + E_W ) \equiv 0$ over $B$ we get
\begin{align*}
K_W\cdot f_W &= ( h^* K_Z + (n-1) F_W ) \cdot f_W \\
&= K_Z \cdot f_Z - (n-1) E_W \cdot f_W \\
&= n-2 - (n-1) = -1.
\end{align*}

Thus $F_W$ can be contracted to a non-singular curve. We get a contraction morphism $\tau: W \longrightarrow V$. By the Hurwitz formula for $h|_{E_W}$ we have $$K_{E_W} = h|_{E_W}^* \left( K_E + \frac{n-1}{n} R \right), \ \ K_E = -(c_1+c_2+c_3)H=-nH$$ where $R\sim b H$ is the ramification divisor, $H$ is the positive generator of $\mathrm{Cl}\ E\simeq \mathbb{Z}$, and $b \in \mathbb{Z}_{\geq 1}$. 

\

Now we go in the other direction. We start from a $\bbmu_n$-del Pezzo fibration $\pi_V: V \longrightarrow B \ni o$ with the following conditions: the central fiber $E_V = \pi_V^{-1}(o)$ is $\bbmu_n$-invariant and has a fixed elliptic curve $C_V$ such that the $\bbmu_n$-action on the projectivization of the normal bundle $\mathbb{P}(N_{C/V})$ is trivial. We blow-up $C_V$ and obtain a $\bbmu_n$-del Pezzo fibration $\pi_W: W\longrightarrow B \ni o$ with the central fiber $E_W + F_W$. Denote the contraction morphism by $\tau: W \longrightarrow V$. By assumption, $\bbmu_n$ fixes $F_W$ pointwise. We take the quotient $h: W\longrightarrow Z$ by the $\bbmu_n$-action. Notice that $h$ is ramified along $F_W$, and $E_W$ is a degree $n$ cover of $h(E_W)=:E$. Now we show that $E$ can be contracted. One checks that any curve in $E$ is $K_Z$-negative. It follows that there is a contraction morphism $\psi: Z \longrightarrow X$ to a terminal del Pezzo fibration $\pi: X \longrightarrow B \ni o$. We claim that the point $x_0:=\psi(E)$ is non-singular on $X$. We consider three cases.
\begin{enumerate}
\item
$d=3$. One checks that $E_W / \bbmu_3 \simeq \mathbb{P}^2$, and $f$ is the blow-down to a non-singular point.
\item
$d=2$. One checks that $E_W / \bbmu_4 \simeq \mathbb{P}(1,1,2)$, and $f$ is the inverse of a weighted blow-up with the weights $(1, 1, 2)$ of a non-singular point.
\item
$d=1$. One checks that $E_W / \bbmu_6 \simeq \mathbb{P}(1,2,3)$, and $f$ is the inverse of a weighted blow-up with the weights $(1, 2, 3)$ of a non-singular point.
\end{enumerate}


\end{computation}

We are ready to prove the following theorem.

\begin{thm}
\label{smoothclassification}
Let $\pi: X \longrightarrow B\ni o$ be a non-singular del Pezzo fibration such that the fiber $F = \pi^{-1}(o)$ is non-rational. Then there is 1-to-1 correspondence between such $\pi$ and $\bbmu_n$-del Pezzo fibrations $\pi_V: V \longrightarrow B\ni o$ with the following properties: 

\begin{itemize}
\item
the central fiber $E_V=\pi_V^{-1}(o)$ is a non-singular $\bbmu_n$-minimal del Pezzo surface of degree $d$, 
\item
the locus of fixed points of $\bbmu_n$ is an elliptic curve $C\subset E_V$,
\item
the action of $\bbmu_n$ on $\mathbb{P}(N_{C/V})$ is trivial.
\end{itemize}

There are only three possible cases (here $d=K_{F}^2$):

\begin{enumerate}
\item \label{ns_case1}
$d=3, \ n=3, \\ 
E_V \simeq ( w^3 = q_3(x,y,z)) \subset \mathbb{P}^3,\\ \bbmu_3: w \mapsto \zeta_3 w, \\
F \simeq ( 0 = q_3(x,y,z)) \subset \mathbb{P}^3;$

\
\item \label{ns_case2}
$d=2,\ n=4, \\ 
E_V \simeq ( w^2 = q_4(x, y) + z^4 ) \subset \mathbb{P}(1,1,1,2), \\ \bbmu_4: z\mapsto \sqrt{-1} z, \\ 
F \simeq ( w^2 = q_4(x, y) ) \subset \mathbb{P}(1,1,1,2);$

\ 
\item \label{ns_case3}
$d=1, \ n=6, \\ 
E_V \simeq ( w^2 = z^3 + \alpha x^4 z + \beta x^6 + y^6 ) \subset \mathbb{P}(1,1,2,3), \\ \bbmu_6: y \mapsto \zeta_6 y, , \ \alpha, \beta \in \mathbb{C}, \\
F \simeq ( w^2 = z^3 + \alpha x^4 z + \beta x^6 ) \subset \mathbb{P}(1,1,2,3).
$
\end{enumerate}
\end{thm}

\begin{proof}
By Proposition \ref{DegreeBound} we have $d\leq 3$. We consider three cases: $d=-E_0^2 = 1, 2, 3$. According to \cite[4.57]{KM-1998}, $\mathrm{mult}_{x_0} F = 3, 2, 2$, respectively. We apply the general construction described above.

\

\textbf{Case $ d = 3 $}. In this case we can take $\psi$ to be the standard blow-up of $x_0$. We have $$K_Z = \psi^* K_X + 2E, \ \ F_Z = \psi^* F - 3 E$$ and $E\simeq \mathbb{P}^2$. By adjunction $K_{F_Z} = \psi|_{F_Z}^* K_F - E|_{F_Z}$, and $F_Z$ is non-singular. By Construction we get a non-singular fibration into cubic surfaces $\pi_V: V \longrightarrow B \ni o$ with the non-singular fiber $E_V = \pi^{-1}(o)$. Moreover, the group $\bbmu_3$ acts on $\pi_V$, and the fixed curve of this action is a non-singular elliptic curve $C_V$. Since $E_V$ is non-singular del Pezzo surface with the action of $\bbmu_3$, we may apply the classification of \cite{DI-2010} and get the case \ref{ns_case1} of the theorem.

\

\textbf{Case $ d = 2 $}. By \cite[4.57]{KM-1998} up to an analytic change of coordinates in the neighbourhood of $x_0$ the fiber $F \subset X$ is given by the equation $$q_4( x, y )+ w^2 = 0,$$ and $\mathrm{mult}_{x_0} q_4 = 4$. Blow up $x_0 \in X$ with the weights $(1,1,2)$ in $x, y, z$. Notice that the blow-up with the weights $(1,1,1)$ leads to a non-normal surface $F_Z$. We get $$K_Z = \psi^* K_X + 3E, \ \ F_Z = \psi^* F - 4 E$$ where $E\simeq \mathbb{P}(1,1,2)$ is the exceptional divisor and $F_Z=\psi^{-1}_*F$. Notice that $F_Z$ is non-singular, and $Z$ has one singular point $p$ of type $\frac{1}{2}(1, 1, 1)$ which corresponds to the unique singular point $p$ of $E$. Put $C = E \cap F_Z$. The curve $C$ does not pass through~$p$.

We apply Construction \ref{ns_construction}. Locally one checks that $h$ is ramified at two points $q_1, q_2 \in W$ such that $\{q_1, q_2\} = h^{-1}(p)$, and that $W$ is non-singular. Using the classification of \cite{DI-2010} we get the case \ref{ns_case2} of the theorem.

\

\textbf{Case $ d = 1 $}. By \cite[4.57]{KM-1998} up to an analytic change of coordinates in the neighbourhood of $x_0$ the fiber $F \subset X$ is given by the equation $$ w^2 + z^3 + z q_4 (x) + q_6(x) =0 $$ where $\mathrm{mult}_{x_0} q_i \geq i$. We blow-up $x_0 \in X$ with the weights $(1,2,3)$ in $x, y, z$. Denote the blow-up morphism by $\psi: Z \longrightarrow X$. We get $$K_Z = \psi^* K_X + 5E, \ \ F_Z = \psi^* F - 6 E$$ where $E\simeq \mathbb{P}(1,2,3)$. Notice that $F_Z$ is non-singular. 

It is easy to see that $Z$ has two singular points $p_1$ and $p_2$ of types $\frac{1}{2}(1, 1, 1)$ and $\frac{1}{3}(1,1,2)$ which correspond to the singular points of $E$. Put $C = E \cap F_Z$. The curve $C$ does not pass through $p_1$, $p_2$. Locally one checks that $h$ is ramified at the preimages of $p_1$ and $p_2$, and that $W$ is non-singular. Using the classification of \cite{DI-2010} we get the case \ref{ns_case3} of the theorem.\end{proof}

\

\section{Ordinary double points}

Suppose that $\pi: X \longrightarrow B\ni o$ is a del Pezzo fibration with singularities that are analytically isomorphic to $(xy+zt=0)\subset \mathbb{C}^4$. Such singularities are called \textit{ordinary double points}. By Remark \ref{reduced} the non-rational fiber $F=\pi^{-1}(o)$ is a reduced irreducible normal Gorenstein surface with a unique simple elliptic singularity $x_0\in F$. 

\begin{proposition}
\label{twocases}
Let $\pi: X \longrightarrow B\ni o$ be a del Pezzo fibration with at worst ordinary double points. Suppose that the central fiber $F=\pi^{-1}(o)$ is non-rational and $X$ has at least one singular point on $F$. Then $F$ is a generalised cone over an elliptic curve and its degree $d=K_F^2$ is equal to either $1$ or~$4$.
\end{proposition}
\begin{proof}The first claim again follows from the classification \cite{HW-1981}. Since~$F$ is Cartier, the point $x_0$ is the only singularity of $X$ on $F$. It corresponds to the vertex of the cone. Consider the standard resolution $\psi: Z \longrightarrow X$ of $x_0$. The exceptional divisor $E$ is isomorphic to $\mathbb{P}^1\times \mathbb{P}^1$. We have  
\begin{align*}
K_Z &= \psi^* K_X + E, \\
F_Z &= \psi^* F - n E, \\
K_{F_Z} &= \psi|_{F_Z}^* K_F - (n - 1)E|_{F_Z}
\end{align*}
where $n\geq 1$. We consider two cases: $n\geq 2$ and $n=1$.

\

\textbf{Case $n\geq 2$}. We show that $n=2$ and $F_Z$ is non-singular. Notice that the exceptional divisors of $\psi|_{F_Z}$ have negative integral discrepancies. Consider the normalization $\nu: \overline{F_Z} \longrightarrow F_Z$. The resulting discrepancies of $\nu \circ \psi|_{F_Z}$ are also negative and integral. Since $F$ has a simple elliptic singularity, any divisor on $\overline{F_Z}$ with negative discrepancy should appear on the minimal resolution $\phi: T \longrightarrow F$. Recall that there is only one $\phi$-exceptional divisor $E_0$, and its discrepancy is $-1$. Hence there is only one $\nu \circ \psi|_{F_Z}$-exceptional prime divisor on $\overline{F_Z}$, and $\nu$ is crepant. Thus $F_Z$ is normal, $E|_{F_Z}$ is reduced, and $F_Z$ is dominated by $T$. Thus $F_Z$ is non-singular and $n=2$. Moreover, $E\cap F_Z$ is a non-singular elliptic curve~$C$. On $E$ it is given by a divisor of bidegree $(2,2)$. 

\

\textbf{Case $n=1$}. Then $F_Z = \psi^*F - E$, and ${F_Z}|_E = - E|_E$, so $E\cap F_Z$ is a divisor of bidegree $(1,1)$ on $E$. In particular, it is reduced. Hence $F_Z$ is normal. Moreover, $E\cap F_Z$ cannot be irreducible: in this case $F_Z$ would be non-singular, but any resolution of $F$ should contain a non-rational exceptional curve. Hence $E\cap F_Z$ is a union of two intersecting lines $L_1$ and $L_2$. The point $p$ of their intersection is singular on $F_Z$. The morphism $\psi|_{F_Z}$ is crepant: $K_{F_Z} = \psi|_{F_Z}^* K_F$. Consider the minimal resolution $\chi: \tilde{F} \longrightarrow F_Z $ and the commutative diagram
\begin{equation}
\vcenter{
\label{resolutions1}
\xymatrix{
F_Z  \ar[d]_{\psi|_{F_Z} } & \tilde{F} \ar[l]_{\chi } \ar[d]^{\eta}  \\
F & T \ar[l]_{\phi } \\
}
}
\end{equation}

The morphism $\eta$ exists since $T$ is the minimal resolution of $F$. 

\

\begin{lem}
\label{simpleelliptic}
The point $p$ is a simple elliptic singularity on $F_Z$, and $\eta$ is the blow-down of two $(-1)$-curves $\chi_*^{-1}L_1$ and $\chi^{-1}_*L_2$.
\end{lem}
\begin{proof}

Suppose that there exists a $\chi$-exceptional curve $E'$ such that $E'\neq \tilde{E_0}:=\eta^{-1}_*E_0$. Since $\chi^{-1}(p)$ is connected we may assume that $E'$ intersects~$\tilde{E}_0$. One checks that $\chi$ is crepant at all $\chi$-exceptional curves except $\tilde{E_0}$ (because $T$ contains only one $\phi$-exceptional curve $E_0$ with negative discrepancy). Since $K_{\tilde{F}}$ is $\chi$-nef we have
$$ 0 \leq K_{\tilde{F}} \cdot E' = ( \chi^* \psi|_{F_Z}^* K_F - \tilde{E_0} ) \cdot E' = - \tilde{E_0} \cdot E' \leq 0.$$

Thus $E'$ does not intersect $\tilde{E_0}$ which is a contradiction. Thus $\tilde{E_0}$ is the unique $\chi$-exceptional curve. It is a non-singular elliptic curve since it dominates $E_0 \subset T$. Clearly, $\chi^{-1}_*L_1$ and $\chi^{-1}_*L_2$ are disjoint $(-1)$-curves. 
\end{proof}

\

We have $K_{\tilde{F}} =\chi^* \psi|_{F_Z}^* K_F - \tilde{E_0}.$ Thus $K_{\tilde{F}}^2 = d + \tilde{E_0}^2$. By the Noether formula we get $K_{\tilde{F}}^2+\chi_{\mathrm{top}}(\tilde{F})=0$. Here $\chi_{\mathrm{top}}(\tilde{F})=2$ since $\tilde{F}$ is a blow-up of two points on the ruled surface~$T$. Thus $K_{\tilde{F}}^2=-2$, and $-\tilde{E_0}^2=d+2$. On the other hand, by \cite[4.57]{KM-1998} we have $-\tilde{E_0}^2 \leq \mathrm{dim}\ \mathrm{T}_{p, Z} = 3$ (recall that $Z$ is non-singular). Hence $d+2 \leq 3$, thus $d=1$ and $E_0^2=-1$.
\end{proof}

We are ready to prove

\begin{thm}
\label{sing_classification}
Let $\pi: X \longrightarrow B\ni o$ be a del Pezzo fibration with at worst ordinary double points. Suppose that the fiber $F = \pi^{-1}(o)$ is non-rational and $X$ has at least one singular point on $F$. Then there is 1-to-1 correspondence between such $\pi$ and (weak and analytic in the case \ref{odp_case2} below) $\bbmu_n$-del Pezzo fibrations $\pi_V: V \longrightarrow B \ni o$ with the following conditions: 

\begin{itemize}
\item
the central fiber $E_V=\pi_V^{-1}(o)$ is a non-singular (weak in the case \ref{odp_case2} below) del Pezzo surface of degree $d$ with $\rho^{\bbmu_2}(E_V)=2$, 
\item
one-dimensional locus of fixed points of $\bbmu_n$ is an elliptic curve $C\subset E_V$,
\item
the action of $\bbmu_n$ on $\mathbb{P}(N_{C/V})$ is trivial.
\end{itemize}
 
There are only two possible cases (here $d=K_{F}^2$):

\begin{enumerate}
\item\label{odp_case1}
$d=4, \ n=2,\ E_V$ has two $\bbmu_2$-conic bundle structures, 
\item\label{odp_case2}
$d=1,\ n=4,\ E_V$ has one $\bbmu_4$-invariant $(-1)$-curve. There exists one $\bbmu_4$-invariant point.
\end{enumerate}
\begin{proof}
By Proposition \ref{twocases} there are two cases two consider: $d=1$ or $4$. 

\

\textbf{Case $d=4$.}
We are in the setting of the first case of Proposition \ref{twocases}. We make the base change. We will construct the following diagram:
\begin{equation}
\vcenter{
\label{basechange}
\xymatrix{
F_W + E_W \subset W \ar[r]^{h} \ar[d]^{\tau} & \ar[d]_{\psi} Z \supset F_Z + 2 E 
\\
E_V \subset V \ar[d]^{\pi_V} & X \supset F \ar[d]_{\pi}
\\
B'  \ar[r]^{\alpha} & B 
}}
\end{equation}
here $B'\simeq B$, $\alpha: t\mapsto t^2$, and $W$ is the normalization of $Z \times_B B'$. As in Construction \ref{ns_construction}, one checks that $W$ is non-singular, $h$ is ramified along $F_W:=h^{-1}(F_Z)$, and the covering map $$h|_{E_W}: h^{-1}(E) =: E_W \longrightarrow E$$ is ramified along a non-singular elliptic curve $E\cap F_Z$. The Galois group $\bbmu_2$ of $h$ acts on $W$. By the Hurwitz formula $E_W$ is a quartic del Pezzo surface. One checks that $F_W$ can be contracted to a non-singular elliptic curve, so we obtain a $\bbmu_2$-equivariant morphism $\tau: W \longrightarrow V$. Hence $\pi_V: V\longrightarrow B \ni o$ is a fibration into quartic del Pezzo surfaces with a non-singular central fiber $E_V$. Notice that $\rho^{\bbmu_2}(E_V) = 2$ since $E_V$ admits two $\bbmu_2$-equivariant conic bundle structures.

If we start from a $\bbmu_2$-del Pezzo fibration $\pi_V: V \longrightarrow B \ni o$ of degree $4$ with the properties as in the theorem, one checks that we can go along the diagram in the other direction and get a del Pezzo fibration $\pi: X \longrightarrow B \ni o$ with a non-rational central fiber and an ordinary double point.

\

\textbf{Case $d=1$.} Let $\psi: Z \longrightarrow X$ be a small resolution of $x_0\in X$. That is, the exceptional locus of $\psi$ is a curve $L\simeq \mathbb{P}^1$ and $Z$ is a non-singular complex manifold. Notice that $F_Z:=\psi^{-1}_*F$ is a singular complex surface. As in Lemma \ref{simpleelliptic} one checks that $F_Z$ has one simple elliptic singularity, say $z_0\in F_Z \subset Z$. Arguing as in Lemma \ref{simpleelliptic} we see that the self-intersection of the exceptional elliptic curve equals $-2$. Let $\psi': Z' \longrightarrow Z$ be the blow-up with the weights $(1,1,2)$. From \cite[4.57]{KM-1998} it follows that $F_{Z'} = \psi'^{-1}_*F_Z$ is the minimal resolution of $F_Z$. We have 
\begin{align*}
K_{Z'} &= \psi'^*K_Z + 3 E', \\
F_{Z'} &= \psi'^*F_Z - 4 E', \\
K_{F_{Z'}} &= \psi'|_{F_{Z'}}^* K_{F_Z} - E'|_{F_{Z'}}
\end{align*}
where $E' \simeq \mathbb{P}(1,1,2)$ and $F_{Z'}=\psi'^{-1}_*F_Z$. Notice that $Z'$ has one singular point of type $\frac{1}{2}(1,1,1)$, and $F_{Z'}$ has one reducible fiber. We will construct the following diagram 
\begin{equation}
\vcenter{
\label{basechange}
\xymatrix{
F_{W'} + E_{W'} \subset W' \ar[dd]^{\tau} & F_W + E_W \subset W  \ar@{-->}[l]_{h'} \ar[r]^{h} & \ar[d]_{\psi'} Z' \supset F_{Z'} + 2 E_{Z'} 
\\
& & Z \supset F_Z \ar[d]_{\psi}
\\
E_V \subset V \ar[d]^{\pi_V} & & X \supset F \ar[d]_{\pi}
\\
B'  \ar[rr]^{\alpha} & & B 
}}
\end{equation}
where $B'\simeq B$, $\alpha: t\mapsto t^4$, and $W$ is the normalization of $Z \times_B B'$. As in the previous case $W$ is non-singular, $h$ is ramified along $F_W:=h^{-1}(F_Z)$, and $h|_{E_W}$ is ramified along a non-singular elliptic curve $E_{W}\cap F_{W}$ where $E_W:=h^{-1}(E_{Z'})$. The Galois group $\bbmu_4$ of $h$ acts on $W$, and the fiber $\pi_W^{-1}(o) = F_W + E_W$ is reduced. By the Hurwitz formula $E_W$ is a degree $2$ del Pezzo surface. One checks that $E_W$ is non-singular. Notice that $F_W\simeq F_{Z'}$ has one reducible fiber $f'_W=f_1+f_2$. Both $f_1$ and $f_2$ are $(-1)$-curves on $F_W$. Without loss of generality, assume that $f_1$ intersects the elliptic curve $C_W:=F_W\cap E_W$. 

We make a flop $h'$ in the curve $f_1$. It is the simplest Atiyah-Kulikov flop, see \cite[4.2]{Ku-1977}. We obtain a threefold $W'$ with the central fiber $E_{W'}+F_{W'}$ where $E_{W'}$ and $F_{W'}$ are the strict transforms of $E_W$ and $F_W$, $E_{W'}$ is the blow-up of a point in $E_W$, and $F_W'$ is the blow-down of $f_1$. Thus, $E_{W'}$ is a non-singular weak (that is $-K_{E_{W'}}$ is nef and big) del Pezzo surface of degree $1$. Then $F_{W'}$ is a ruled surface that can be contracted onto a curve, and we get a degree $1$ del Pezzo fibration $\pi_V: V \longrightarrow B \ni o$. 

If we start from a $\bbmu_4$-del Pezzo fibration $\pi_V: V \longrightarrow B \ni o$ of degree $1$ with the properties as in the theorem, one checks that we can go along the diagram in the other direction and get a del Pezzo fibration $\pi: X \longrightarrow B \ni o$ with a non-rational central fiber and an ordinary double point.
\end{proof}


\end{thm}

\


\def\cprime{$'$} \def\mathbb#1{\mathbf#1}

\Addresses

\end{document}